\documentclass{amsart}

\usepackage[latin1]{inputenc}
\usepackage{amsfonts}
\usepackage{amsmath}
\usepackage{amsthm}
\usepackage{amssymb}
\usepackage{latexsym}
\usepackage{enumerate}

\newtheorem{theorem}{Theorem}%[section]
\newtheorem{lemma}[theorem]{Lemma}
\newtheorem{proposition}[theorem]{Proposition}

\newtheorem{question}[theorem]{Question}

\newtheorem{definition}[theorem]{Definition}

\numberwithin{equation}{section}

\begin{document}

\newcommand{\cc}{\mathfrak{c}}
\newcommand{\N}{\mathbb{N}}
\newcommand{\BB}{\mathbb{B}}
\newcommand{\C}{\mathbb{C}}
\newcommand{\Q}{\mathbb{Q}}
\newcommand{\G}{\mathbb{G}}
\newcommand{\R}{\mathbb{R}}
\newcommand{\Z}{\mathbb{Z}}
\newcommand{\T}{\mathbb{T}}
\newcommand{\st}{*}
\newcommand{\PP}{\mathbb{P}}
\newcommand{\rin}{\right\rangle}
\newcommand{\SSS}{\mathbb{S}}
\newcommand{\forces}{\Vdash}
\newcommand{\supp}{\text{supp}}
\newcommand{\dom}{\text{dom}}
\newcommand{\osc}{\text{osc}}
\newcommand{\F}{\mathcal{F}}
\newcommand{\A}{\mathcal{A}}
\newcommand{\B}{\mathcal{B}}
\newcommand{\D}{\mathcal{D}}
\newcommand{\I}{\mathcal{I}}
\newcommand{\X}{\mathcal{X}}
\newcommand{\Y}{\mathcal{Y}}
\newcommand{\CC}{\mathcal{C}}
\newcommand{\Car}{\mathfrak{1}}
\newcommand{\sss}{\rm SSW}
\newcommand{\sw}{\rm SW}
\newcommand{\n}{{2, \A}}

\thanks{The research of the authors was supported by the NCN (National Science
Centre, Poland) research grant no.\ 2020/37/B/ST1/02613.}

\title[Equilateral and separated sets]
{Equilateral and separated sets in some  Hilbert generated Banach spaces }

\author{Piotr Koszmider}
\address{Institute of Mathematics, Polish Academy of Sciences,
ul. \'Sniadeckich 8,  00-656 Warszawa, Poland}
\email{\texttt{piotr.math@proton.me}}

\author{Kamil Ryduchowski}
\address{Faculty of Mathematics, Informatics and Mechanics, 
University of Warsaw, ul. Banacha 2, 02-097 Warszawa, Poland}
\address{Institute of Mathematics, Polish Academy of Sciences,
ul. \'Sniadeckich 8,  00-656 Warszawa, Poland}
\email{\texttt{kryduchowski@impan.pl}}

\begin{abstract} 
We study Hilbert generated versions of nonseparable Banach spaces $\X$ considered by Shelah, Stepr\=ans 
and Wark where the behavior of the norm on nonseparable subsets is so irregular that it does not allow
 any linear bounded operator on $\X$ other than 
a diagonal operator (or a scalar multiple of the identity) plus a separable range operator.
We address the questions if  these spaces
admit uncountable equilateral sets and if their unit spheres admit uncountable
$(1+)$-separated or $(1+\varepsilon)$-separated sets. 
 We resolve some of the above questions for  two types of these spaces  by showing
both absolute and undecidability results. 

The corollaries are that the continuum  hypothesis (in fact: the existence of a nonmeager
set of reals of the first uncountable cardinality)  implies the existence of
an equivalent renorming of the nonsepareble Hilbert space $\ell_2(\omega_1)$ which does
not admit any uncountable equilateral set and it implies the existence of 
 a nonseparable Hilbert generated Banach space containing an isomorphic copy of
 $\ell_2$ in each nonseparable subspace,  whose unit sphere does not admit
an uncountable equilateral set and does not admit an uncountable $(1+\varepsilon)$-separated set for any $\varepsilon>0$.
This could be compared with a recent result by H\'ajek, Kania and Russo 
saying that all nonseparable reflexive Banach spaces  admit uncountable
 $(1+\varepsilon)$-separated sets in their unit spheres.

\end{abstract}

\maketitle

\section{Introduction}

%If $\X$ is a Banach space,  $S_\X$ stands for its unit sphere. 
 \begin{definition}\label{def-sep} Let $\X$ be a Banach space and $\delta>0$.
\begin{itemize}
\item $\Y\subseteq \X$ is $(\delta+)$-separated if $\|y-y'\|>\delta$ for distinct $y, y'\in \Y$.
\item $\Y\subseteq \X$ is $\delta$-separated\footnote{By the classical Riesz lemma for any $\varepsilon>0$ 
any infinite dimensional 
Banach space $\X$ of density $\kappa$
admits a $(1-\varepsilon)$-separated set $\Y$ of the unit sphere of $\X$ of cardinality $\kappa$.}
 if $\|y-y'\|\geq\delta$ for distinct $y, y'\in \Y$.
\item $\Y\subseteq \X$ is $\delta$-equilateral\footnote{Note that    a Banach space admits 
 an  equilateral set of an infinite cardinality $\kappa$ if and only if its unit sphere
admits a $1$-equilateral set of cardinality $\kappa$ 
(scale the original equilateral set and translate one of its points to $0$).} if $\|y-y'\|=\delta$ for distinct $y, y'\in \Y$. 
\end{itemize}
\end{definition}
Given a Banach space $\X$ of a prescribed class much research has been done attempting
to determine
the existence of large equilateral sets of $\X$ or  large $(1+)$-separated  or $(1+\varepsilon)$-separated sets of the unit sphere 
of $\X$ denoted by $S_\X$.  These  questions are of isometric nature but many results depend on the isomorphic class
of the Banach space.

If $\X$ is separable (improving a theorem of Kottman \cite{kottman}), Elton and Odell proved that $S_\X$ contains
an infinite $(1+\varepsilon)$-separated set for some $\varepsilon>0$ (\cite{elton-odell}). 
If $\X$ is nonreflexive $\varepsilon$ can be assumed to be any number smaller than $\sqrt[5]{4}-1$ by a 
result of Kryczka and Prus (\cite{kryczka-prus})
and it can be estimated for $\X$ being uniformly convex (\cite{neerven}).
Mercourakis and Vassiliadis proved that
any Banach space containing an isomorphic copy of $c_0$ admits an infinite
equilateral set and Freeman, Odell, Sari, and Schlumprecht proved that all uniformly smooth Banach spaces
admit an infinite
equilateral sets. However,
Terenzi in \cite{terenzi}  gave an equivalent renorming of $\ell_1$
which does not admit any infinite equilateral set.  It should be added that every infinite dimensional 
Banach space has an equivalent renorming which admits an infinite equilateral set (\cite{mer-pams}).

To investigate the existence of uncountable  sets in $S_\X$ as in Definition \ref{def-sep} depending
on the properties of the Banach space $\X$ one needs to move
to nonseparable  Banach spaces, as all these sets are  norm discrete.
Here an uncountable version of the Elton-Odell
theorem fails as witnessed by $c_0(\omega_1)$ (\cite{elton-odell}).
Also for such spaces an uncountable version of the Kottman's theorem fails
 (\cite{pk-kottman}) or even the unit spheres can be unions of countably
many sets of diameters strictly less than $1$ (\cite{pk-ad}).  On a more positive
side Kania and Kochanek  (\cite{tt}) have proved that the unit spheres of 
all nonseparable spaces of the form $C(K)$ for $K$ compact Hausdorff
admit an uncountable $(1+)$-separated set but whether they
admit an uncountable  $(1+\varepsilon)$-separated or
an uncountable equilateral set turned out to be  undecidable (\cite{pk-ck}).
Also equivalent renormings of $\ell_1([0,1])$ without infinite
equilateral sets and of $C_0(X)$ spaces  for $X$ locally compact Hausdorff with no uncountable equilateral sets have been
constructed in \cite{pk-wark, pk-kottman}.

However, we have a strong isomorphic positive result of H\'ajek, Kania and Russo saying that all
reflexive nonseparable Banach spaces do admit uncountable
$(1+\varepsilon)$-separated sets.  Also no reflexive nonseparable Banach space with no
(infinite) uncountable equilateral set has been  known so far.  
The following is
our main result which sheds some further light on the relation between the reflexivity and 
uncountable equilateral and separated sets.

\begin{theorem}\label{main}
Suppose that there is a nonmeager set of reals
of the first uncountable cardinality $\omega_1$ (e.g., assume the continuum hypothesis).
\begin{enumerate} 
\item There is an equivalent renorming of $\ell_2(\omega_1)$ which does not
admit an uncountable equilateral set.
\item There is a nonseparable Hilbert generated Banach space
where every nonseparable subspace contains an isomorphic copy of $\ell_2$,
whose unit sphere does not admit an uncountable equilateral set nor 
an uncountable $(1+\varepsilon)$-separated set for any $\varepsilon>0$.
\end{enumerate}
\end{theorem}
\begin{proof}
For (1) use Propositions \ref{hilbert}, \ref{nonmeager} and \ref{hilbert-equi}. For (2) use Theorem \ref{main-ssw} (1).
\end{proof}
 Recall that a Banach space $\X$  is Hilbert generated if there is a bounded linear operator from a Hilbert space
into $\X$ which has dense range.  These are exactly the  Banach spaces whose duals balls are uniform Eberlain compacta
with respect to weak$^*$-topology. See, e.g., \cite{ug, integration, hilbert-gen, hajek-adv} for 
a wide range of results on this class of spaces. Note that the property of having a copy of $\ell_2$ in every
nonseparable subspace is related to being somewhat reflexive (\cite{somewhat}).
 We also note that all consistent or absolute examples of nonseparable Banach spaces
 not admitting uncountable equilateral sets presented in the literature  so far were not    even WLD\footnote{
Spaces of \cite{pk-ck} are $C(K)$ for $K$ separable and spaces of 
\cite{pk-kottman} and \cite{pk-ad} are 
isomorphic to a $C(K)$ for $K$ separable, while  nonseparable WLD 
space cannot be isomorphic to a $C(K)$ space for $K$ separable since
separable Corson compacta are metrizable. The spaces of \cite{pk-wark} are renormings of $\ell_1([0,1])$
hence not WLD as well.}, a class much wider than Hilbert generated spaces. We do not know if  the statements
of Theorem~\ref{main}
 can be obtained  without any additional set-theoretic hypothesis  but 
 we prove in Theorem \ref{main-ssw} (3)  that such a hypothesis
 cannot be removed  in the case of  the concrete class of  spaces ($\sw$-spaces)
 where our consistent example from Theorem \ref{main} (2) belongs.

Our approach to the question of the existence of uncountable sets  as in Definition \ref{def-sep}
in reflexive or ``close to reflexive'' Banach spaces was to examine two types of concrete
classes of Banach spaces considered first   by Shelah in \cite{shelah-graph} and by Shelah and Stepr\=ans  in \cite{ss}
and later modified by Wark \cite{wark, wark-convex}.  We call Banach spaces  of these classes $\sw$-spaces
and $\sss$-spaces (see Section 2 for definitions). This paper is entirely devoted to these spaces
and Theorem \ref{main} is a corollary to the results concerning them.
The reason we started to investigate these
spaces  in the above context  is the ``chaotic'' behavior of the norm  on nonseparable sets (which is
responsible for the rigidity of the spaces in the sense of having few operators -- the original interest in these spaces)
 and the fact that on the other hand, Wark has obtained  their Hilbert generated
 versions (and through the interpolation method, even reflexive and uniformly convex versions).

More specifically, Theorem \ref{main} is
a consequence of item (1) of the following result whose remaining parts show some limitations on how  
$\sw$-spaces and $\sss$-spaces can serve as 
(counter)examples in the above context.

\begin{theorem}\label{main-ssw} $ $
\begin{enumerate}
\item The existence of a nonmeager set of reals
of the first uncountable cardinality $\omega_1$  implies that there is an $\sw$-space where every nonseparable subspace contains
an isomorphic copy of $\ell_2$, which does not admit
any uncountable equilateral set
and whose unit sphere does not admit
an uncountable $(1+\varepsilon)$-separated set for any $\varepsilon>0$.
\item The unit spheres of all $\sw$-spaces and all $\sss$-spaces admit uncountable $(1+)$-separated sets.
\item Martin's axiom and the negation of the continuum hypothesis imply that the unit spheres of 
all $\sss$-spaces admit an uncountable 
$(1+\varepsilon)$-separated set for some $\varepsilon>0$ and 
that all unit spheres of $\sw$-spaces admit uncountable $\sqrt 2$-equilateral sets.
\item There is an  $\sw$-space and an $\sss$-space whose unit spheres admit uncountable $\sqrt 2$-equilateral sets
and ${{\sqrt2+\sqrt5}\over{\sqrt2+1}}$-equilateral sets, respectively.

\end{enumerate}
\end{theorem}
\begin{proof}
For (1) use Propositions \ref{equi-sep}, \ref{somewhat}.
For (2) use
Proposition \ref{1+}.
For (3) use Proposition \ref{ma-equi}. For (4) use Proposition \ref{zfc-equi}.
\end{proof}

The structure of the paper is as follows. Next Section 2 contains preliminaries. Section 3
is devoted to proving items (2) - (4) of Theorem \ref{main-ssw}. Section 4 includes the 
propositions needed for  Theorem \ref{main-ssw} (1). The last Section 5 lists some
questions.

The first uncountable cardinal is denoted by $\omega_1$.
The set
of all $x\in \R^{\omega_1}$ such that $\{\alpha<\omega_1: x(\alpha)\not=0\}$ is at most finite
is denoted by $c_{00}(\omega_1)$. 
If $\X$ is a Banach space, 
$B_\X$ denotes the unit ball of $\X$ and 
$S_\X$ denotes the unit sphere  of $\X$.
By a subspace of a Banach space we always mean a closed linear subspace.
%If $T$ is a linear bounded operator, $T^*$ is its adjoint.
For $x\in \R^{\omega_1}$ the set 
$\supp(x)=\{\alpha\in \omega_1: x(\alpha)\not=0\}$ 
is called the support of $x$. The symbol
$\Car_A$ stands for the characteristic function of $A$. If $A, B$ are disjoint sets,
then $A\otimes B=\{\{\alpha, \beta\}: \alpha\in A, \beta\in B\}$.
If $X$ is a set,  $[X]^{<\omega}$ 
denotes the family of all finite subsets of $X$ and $[X]^n$ denotes the family of all $n$-element subsets of $X$.

\section{Families, norms, colorings and axioms}

\begin{definition}\label{def-normA} Let $\A$ be a family of finite subsets of $\omega_1$ which  is closed under taking subsets
and contains all singletons. We define the norm $\|\ \|_\A$ on $c_{00}(\omega_1)$
by
$$\|x\|_\A=\sup_{A\in \A}\sqrt{\sum_{\alpha\in A}x(\alpha)^2}$$
for all $x\in c_{00}(\omega_1)$. The completion of $c_{00}(\omega_1)$ with respect to this norm
is denoted $(\X_\A, \|\ \|_\A)$. 
\end{definition}

One can prove by standard methods that the set of all $x\in \R^{\omega_1}$ such that $\|x\|_\A$ is finite
is a complete space with respect to $\|\ \|_\A$ and so we have a concrete representation of
$\X_\A$ as the closure of $c_{00}(\omega_1)$ with respect to this norm as in \cite{ss}.  We will work with this representation.
When $\omega_1$ is replaced by $\N$ and the $\ell_2$ norm is replaced by $\ell_1$ norm such spaces
are called combinatorial Banach spaces or generalized Schreier spaces and can be traced back
to Schreier's paper \cite{schreier}. The first use of $\omega_1$ instead of $\N$ seems
to be \cite{shelah-graph} and the first use of the $\ell_2$ norm in this nonseparable context 
seems to be \cite{wark}.  The $\ell_2$ is more cumbersome in calculations but has some better
properties:

\begin{proposition}\label{hilbert}
Suppose that $\A$ is as in Definiton \ref{def-normA}. Then $(\X_\A, \|\ \|_\A)$ is a Hilbert
generated Banach space and the norm $\|\ \|_\n = \|\ \|_2 + \|\ \|_\A$ is an equivalent norm on the Hilbert space 
$(\ell_2(\omega_1), \|\ \|_2)$. 
\end{proposition}
\begin{proof}
Note that for any $x\in c_{00}(\omega_1)$ we have
$\|x\|_\infty \leq \|x\|_\A \leq \|x\|_2$ and so $\|\ \|_2\leq \|\ \|_2 + \|\ \|_\A\leq 2 \|\ \|_2$
and hence $\|\ \|_\n$ and $\|\ \|_2$ are equivalent on $\ell_2(\omega_1)$.
Also the identity operator $\ell^2(\omega_1) \rightarrow \X_\A$ is bounded and has  dense image
and so $\X_\A$ is a Hilbert generated Banach space. 
\end{proof}

\begin{definition} Suppose that $A\subseteq\omega_1$. $P_A\colon\X_\A\rightarrow \X_\A$ is defined by 
$$
  P_A(x)(\eta) =
  \begin{cases}
    x(\eta) &  \text{if} \ \eta\in A \\
   0 &  \text{if}\  \eta\not\in A
  \end{cases}
$$
\end{definition}

\begin{lemma}\label{projection-norm} Suppose that $\A$ is  as in Definition \ref{def-normA}.
Suppose that $\emptyset \neq A\subseteq B\subseteq\omega_1$ and $x\in \X_\A$. Then
$\|P_A(x)\|_{\A}\leq \|P_B(x)\|_{\A}$. In particular, $P_A$ is a projection of norm one on $\X_\A$.
\end{lemma}

Now we recall the definitions concerning the other type of norm we will consider which
in the separable setting can be traced back to James (\cite{james}).
By a set of consecutive pairs we mean a set $D$ of doubletons of $\omega_1$  such that
$\max(a)<\min(b)$ or $\max(b)<\min(a)$ for any distinct $a, b\in D$.

\begin{definition}\label{def-nu} Let $\D$ be a  family of finite sets of consecutive pairs containing all single pairs of $\omega_1$.
For $x\in c_{00}(\omega_1)$ define the norm $\nu_\D$ on $c_{00}(\omega_1)$ by
$$\nu_\D(x)=\sup_{D\in \D}\sqrt{\sum_{\{\alpha, \beta\}\in D}|x(\alpha)-x(\beta)|^2}$$
for all $x\in c_{00}(\omega_1)$.% The completion of $c_{00}(\omega_1)$ with  respect to this norm  is denoted 
%by $(\X_{\D}, \|\ \|_\D)$. 
If $\A$ is as in Definition \ref{def-normA}, then the completion of $c_{00}(\omega_1)$
 with  respect to the norm $\|\ \|_{\A,\D}$ defined as $\nu_\D+\|\ \|_\A$
 is denoted by $(\X_{\A, \D}, \|\ \|_{\A, \D})$. 

\end{definition}
The condition requiring all single pairs to be in $\D$ guarantees that $\nu$ is not only a seminorm but
is a norm on $c_{00}(\omega_1)$.  Such seminorms seem to be first considered
in the nonseparable context in \cite{ss} (in fact, the concrete $\nu$ considered in \cite{ss} is a norm due to the above
property).  As in the case of $\X_\A$ we will work with the concrete representation of
$\X_{\A, \D}$, namely, the closure of $c_{00}(\omega_1)$ with respect to $\|\ \|_{\A, \D}$ as in \cite{ss}.

\begin{proposition}Let $\A$ and $\D$ be as in Definitions \ref{def-normA} and \ref{def-nu}. Then
$(\X_{\A,\D}, \|\ \|_{\A, \D})$ is a Hilbert generated Banach space.
\end{proposition}
\begin{proof}
Note that for any $x\in c_{00}(\omega_1)$ and $\A$, $\D$ as in Definitions \ref{def-normA} and \ref{def-nu} we have
$\|x\|_\infty \leq \|x\|_\A \leq \|x\|_2$ and $\|x\|_\infty \leq \nu_D(x)\leq 2\|x\|_2$.
In particular the identity operator  $\ell^2(\omega_1) \rightarrow \X_{\A, \D}$ has  dense image
and so $\X_{\A,\D}$ is a Hilbert generated Banach space.
\end{proof}

In  \cite{shelah-graph} and \cite{ss}  Banach spaces
where the norm is defined using $\ell_1$ norm instead of $\ell_2$
as in our Definitions \ref{def-normA}, \ref{def-nu} are investigated
and  special properties of families $\A, \D$
 are considered.  It
 is shown there that for such special families $\A, \D$  the only linear bounded operators
on the first type of spaces are of the form diagonal with respect to the natural basis plus a separable range operator
and the only linear bounded operators
on the second  type of spaces are 
of the form scalar multiple of the identity   plus a separable range operator.
In \cite{wark} these proofs were modified so that similar results were proved, in our terminology, for
$\X_\A$  and $\X_{\A, \D}$ obtained from the special  families $\A, \D$.
The additional properties of $\A,  \D$ which guarantee few operators in the above sense are defined
(in our terminology) in the following two definitions which introduce two classes
of spaces which are the subject of this paper.

\begin{definition}\label{def-T-families}
Families $\A, \B$ of finite subsets of $\omega_1$ are called  $T$-families
if they are both  closed under taking subsets, $\bigcup\A=\omega_1$,
$\A\cap \B\subseteq[\omega_1]^1\cup\{\emptyset\}$ and for all
families $\{\{a(\alpha), b(\alpha)\}: \alpha<\omega_1\}$ 
of disjoint pairs of elements of $\omega_1$ and $k\in \N$ there are $\alpha_1<\dots<\alpha_k<\omega_1$
such that
\begin{itemize}
\item $\{a(\alpha_1), \dots a(\alpha_k)\}\in \A$,
\item$\{b(\alpha_1), \dots b(\alpha_k)\}\in\B$.
\end{itemize} 
A family  $\A$ of finite subsets of $\omega_1$ is called  a $T_0$-family if
there is a family $\B$ such that 
$\A, \B$ are $T$-families.
If $\A$ is a $T_0$-family then the Banach space $(\X_\A, \|\ \|_\A)$
is called an $\sw$-space.
\end{definition}

\begin{definition}\label{def-D}
A family $\D$ of sets of consecutive pairs is called 
nice with respect to $T$-families $\A, \B$ if 
\begin{enumerate}
\item Whenever $A\in {\mathcal{A}}\cup \B$ and
$D\in {\mathcal{D}}$, then $|A\cap(\bigcup D)|\leq 2$.
\item If $D, D'\in{\mathcal{D}}$ are distinct, then there may
be at most five pairs in $D$ which
 intersect other than itself  pair from $D'$, i.e.,
$$|\{a\in D: a\cap (\bigcup (D'\setminus \{a\}))\not=\emptyset \}|\leq 5 .$$
\item Whenever
$\{d_\xi\colon\xi<\omega_1\}\subseteq[X]^2$  is a collection of consecutive
pairs  and $k\in \N$, then there are $\xi_1<\ldots 
<\xi_k<\omega_1$ such that
$\{d_{\xi_i}\colon 1\leq i\leq k\}\in {\mathcal{D}}$.
\end{enumerate}
If $\A, \B$ are $T$-families and $\D$ is a nice family with respect to them,
then $(\X_{\A, \D}, \|\ \|_{\A, \D})$ is called an $\sss$-space. 
\end{definition}

The constructions of 
$T$-families and a nice family with respect to them 
in \cite{shelah-graph} and \cite{ss} are 
of the form
$\A_c, \B_c, \D_c$ for certain functions $c$ obtained by Todorcevic (\cite{acta}, see \cite{velleman} for a simple construction).
These notions are the subject of the following definition which provides a variety of typical
$\sw$-spaces and $\sss$-spaces.

\begin{definition}\label{T-coloring}
A function $c=(c_0,c_1)\colon [\omega_1]^2\rightarrow I\times J$, where $0,1\in I$ and $J\neq \emptyset$, is called a $T$-coloring
if given any uncountable pairwise disjoint collection $\{\{a_\xi(0), a_\xi(1)\}:\xi<\omega_1\}$ of pairs of $\omega_1$
and $(i_0, j_0), (i_1,j_1)\in I\times J$ there are $\xi<\eta<\omega_1$ such that
$c(\{a_\xi(0), a_\eta(0)\})=(i_0,j_0)$ and $c(\{a_\xi(1), a_\eta(1)\})=(i_1,j_1)$. We define
\begin{itemize}
\item $\A_c=\{a\in [\omega_1]^{<\omega}: c_0[[a]^2]\subseteq \{0\}\}$
\item $\B_c=\{a\in [\omega_1]^{<\omega}: c_0[[a]^2]\subseteq \{1\}\}$
\end{itemize}
In case of $I\times J=3\times[\omega_1]^{<\omega}$ we define $\D_c$ as the collection of 
all such consecutive pairs $\{\xi_1, \eta_1\}, \dots, \{\xi_k, \eta_k\}$ 
with $\xi_i<\eta_i$ for $1\leq i\leq k$ and some $k\in \N$ such that for every $1\leq i< j\leq k$ we have 
$$c(\{\xi_i, \xi_j\})=c(\{\eta_i, \eta_j\})=(2, \{\xi_l, \eta_l: l<i\}).$$
\end{definition}

Summing up the above discussion we have:

\begin{proposition}
$  $
\begin{enumerate}
\item For any sets $I, J$ satisfying $0,1\in I$ and $1\leq |I|, |J|\leq \omega_1$ there exist $T$-colorings 
 $c\colon [\omega_1]^2\rightarrow I\times J$ 
(\cite{acta}).
\item If $c$ is any $T$-coloring, then
$\A_c, \B_c$ are $T$-families. 
\item If $c\colon [\omega_1]^2\rightarrow 3\times[\omega_1]^{<\omega}$ is a $T$-coloring, 
then $\D_c$ is  nice with respect to $\A_c$ and $\B_c$ (\cite{ss}).
\item If $\A$ is any $T_0$-family, then
all linear bounded operators on $\X_\A$ are diagonal plus a separable range operator (\cite{ss, wark}) .
\item If $\A, \B$ are any $T$-families and $\D$ is any family nice with respect to them, then
all linear bounded operators on $\X_{\A, \D}$ are multiples of the identity plus a separable range operator (\cite{ss, wark}) .
\end{enumerate}
\end{proposition}

We will not exploit (4) and (5) above in this paper but they provide the motivation
for investigating the above spaces in the context of irregular behavior of the norm.
 In Section 4 we will  consider $T$-colorings with the following additional properties
which allow us to obtain (2) in Theorem \ref{main} 
for the corresponding $\sw$-spaces $X_{\A_c}$.

\begin{definition}\label{sT-coloring}
A function $c=(c_0,c_1)\colon [\omega_1]^2\rightarrow I\times J$, where $0,1\in I$ and $J\neq \emptyset$, 
is called a strong $T$-coloring
if it is a $T$-coloring and if given  any uncountable pairwise disjoint collection 
$\{A_\xi:\xi<\omega_1\}$ of finite subsets of $\omega_1$
there are $\xi<\eta<\omega_1$ such that
$c_0[A_\xi\otimes A_\eta] =\{0\}$ and there are $\xi<\eta<\omega_1$
 such that $c_0[A_\xi \otimes A_\eta] =\{1\}$.
If $T$ is a strong $T$-coloring, then $\A_c$ is called a strong $T_0$-family
and  $\A_c$ and $\B_c$ are called strong $T$-families.
\end{definition}

The following should be clear from Definitions  \ref{T-coloring} and \ref{sT-coloring}.

\begin{lemma}\label{sT-family}
If $\A$ is a strong $T_0$-family then for every uncountable pairwise disjoint family  $\F$ of 
finite subsets of $\omega_1$
there are distinct $A, B\in \F$ such that
$(A\otimes B)\cap \A=\emptyset$ and there are distinct $A, B\in \F$ such that
$(A\otimes B)\subseteq \A.$ Moreover, any finite set $C\subseteq \omega_1$ such that $[C]^2 \subseteq \A$
 is itself an element of $\A$.
\end{lemma}

For the existence of strong $T$-colorings we need some extra 
set-theoretic hypothesis. Clearly, the following one  is implied by the continuum hypothesis
and is compatible with a variety of its negations (see, e.g., \cite{bj}):

\begin{proposition}[\cite{luzin}]\label{nonmeager} If there is a nonmeager
subset of $\R$ of the first uncountable cardinality $\omega_1$, then there exists
a strong $T$-coloring $c\colon[\omega_1]^2\rightarrow3\times[\omega_1]^{<\omega}$.
\end{proposition}

In Proposition \ref{ma-equi} we will consider another well known additional 
set-theoretic axiom, namely Martin's axiom.
For the definition and related terminology the reader may see any set theory textbook, e.g., \cite{bj}. 
 However, the following remarks
should be sufficient to follow the proof of this proposition. Two elements $p, q$ of a partial order $(\PP, \leq)$
are said to be compatible if there is $r\in \PP$ such that $r\leq p, q$. An antichain in $\PP$ is any pairwise incompatible 
subset of $\PP$. $\PP$ is c.c.c. if it does not admit an uncountable antichain.
A subset $\PP'$ of $\PP$ is centered if for every finite $p_1, \dots, p_k\in \PP'$ there is $p\in \PP$
such that $p\leq p_1, \dots, p_k$. We will use the following fact:

\begin{proposition}[1.4.24 \cite{bj}]\label{ma-sigma}
Assuming Martin's axiom and the negation of the continuum hypothesis every partial
order $\PP$ of cardinality $\omega_1$ which is c.c.c. is $\sigma$-centered, that is $\PP=\bigcup_{n\in \N}\PP_n$
where each $\PP_n$ is centered.
\end{proposition}

\section{General $\sw$-spaces and $\sss$-spaces}

%Note that $\{\Car_{\{\alpha\}}:\alpha<\omega_1\}$ is a $1$-separated 
%set in $S_{\X_\A}$.

\begin{proposition}\label{1+}
The unit sphere of every $\sw$-space and every $\sss$-space admits an  uncountable $(1+)$-separated set.
\end{proposition}
\begin{proof}
Let $\A$ and $\B$ be $T$-families inducing an $\sw$-space $\X_\A$. 
Note that every uncountable subset of $\omega_1$ has a two element subset in $\A$. 
This follows from Definition \ref{def-T-families}. On the other hand the empty set and all
singletons are in $\A$, as $\A$ is closed under subsets and covers $\omega_1$.

Proceeding by induction we construct a family $(A_\alpha)_{\alpha<\omega_1}$ of countable
subsets of $\omega_1$ as follows: 
$A_0 \subseteq \omega_1$ is any maximal (necessarily at most countable) set 
satisfying $[A_0]^{<\omega}\cap\A\subseteq[A_0]^1\cup\{\emptyset\}$.
 Having constructed $A_\alpha$ for $\alpha<\beta$ we pick $\gamma > 
\sup_{\alpha<\beta} \sup A_\alpha$ and we take  as $A_\beta$ any maximal subset of $
\omega_1 \setminus \gamma$ satisfying  $[A_\beta]^{<\omega}\cap\A\subseteq[A_\beta]^1\cup\{\emptyset\}$
Possibly $A_\alpha$s 
 may be singletons, but they are all nonempty.
 Having constructed $A_\alpha$s let $x_\alpha$ be any element of ${\ell_2(\omega_1)}$ 
satisfying $\supp(x_\alpha)=A_\alpha$, 
$x_\alpha(\min (A_\alpha)) = 1$ and $|x_\alpha(\gamma)|\leq 1$
 for $\gamma\neq \min (A_\alpha)$, $\gamma<\omega_1$.  By the choice
of $A_\alpha$ we have $x_\alpha$s in $\X_\A$. 
Moreover, since the only sets of $\A$ which meet $A_\alpha$ are at most singletons
we have that $\|x_\alpha\|_\A=1$ for every $\alpha<\omega_1$.

To see that $(x_
\alpha)_{\alpha<\omega_1}$ forms an uncountable $(1+)$-separated set in $\X_\A$ pick any $\alpha<\beta <
\omega_1$. It follows from the maximality of $A_\alpha$ that  there is $\xi\in A_\alpha$
such that  $\{\xi, \min(A_\beta)\}\in \A$. Therefore $$\| x_\alpha - x_\beta \|_{\A} \geq 
\sqrt{|x_\alpha(\xi)|^2 + |x_\beta(\min A_\beta)|^2} > 1, \leqno (*)$$
which concludes the proof in the case of $\sw$-spaces.

Now suppose that $\D$ is a nice family with respect to $\A$ and $\B$ and consider an $\sss$-space $\X_{\A,\D}$.
 First let us note that above we can pick $x_\alpha$s so that the numbers $\nu_D(x_\alpha)$ are rational. 
 Indeed, if $y\in \ell_2(A_\alpha)$ is any vector satisfying $\supp(y) = A_\alpha \setminus \{ \min(A_\alpha)\}$ 
 and $|y(\gamma)|\leq 1$ for $\gamma\in A_\alpha$ 
 (so that above we picked $x_\alpha$ to be $\Car_{\{\min(A_\alpha)\}}+y$), then the function which sends $t\in [0,1]$ to
$\nu_D(\Car_{\{\min(A_\alpha)\}}+ty)$ is  continuous and defined on a connected domain.
If it is constant, then $\nu_D(\Car_{\{\min(A_\alpha)\}}+y)=\nu_D(\Car_{\{\min(A_\alpha)\}})=1$. If it is nonconstant
it assumes some rational value.
So we will succeed in obtaining such a value of $\nu_\D(x_\alpha)$.

For $n\in \N$ let $D_{n, \alpha}\in \D$ be such
that 
$$\nu_\D(x_\alpha)=\sup_{n\in \N}\sqrt{\sum_{\{\xi, \eta\}\in D_{n, \alpha}}|x_\alpha(\xi)-x_\alpha(\eta)|^2}.$$
By passing to an uncountable subsequence we may assume that 
for every $\alpha<\beta<\omega_1$ we have 
$$\bigcup\bigcup \{D_{n, \alpha}:   n\in \N\}\cap A_\beta=\emptyset.$$
This implies that 
$$\nu_\D(x_\alpha)\leq\nu_\D(x_\alpha-x_\beta)\leqno (**)$$
as the support of $x_\beta$ is equal to $A_\beta$.
It follows from (*) and (**)  that for every $\alpha<\beta<\omega_1$ we have 
$$\| x_\alpha - x_\beta \|_{\A, \D} >
\|x_\beta\|_\A+\nu_\D(x_\alpha).\leqno (***)$$
By passing to an uncountable subset we may assume that
the values of $\nu_D(x_\alpha)$ are the same as they are rational, and so
consequently all the values of $\|x_\alpha\|_{\A, \D}$ are the same as well, since $\|x_\alpha\|_\A=1$
for every $\alpha<\omega_1$.
So (***) gives $\| x_\alpha - x_\beta \|_{\A, \D} >
\|x_\beta\|_\A+\nu_\D(x_\beta)=\|x_\beta\|_{\A, \D}$ for every $\alpha<\beta<\omega_1$.
That is, the set $\lbrace \frac{1}{\|x_\alpha\|_{\A, \D}} x_\alpha:\ \alpha<\omega_1\rbrace$ is $(1+)$-separated.
\end{proof}

\begin{proposition}\label{ma-equi}
Assume Martin's axiom and the negation of the continuum hypothesis.
The unit sphere of every  $\sw$-space  contains an uncountable $\sqrt{2}$-equilateral set and the 
unit sphere of every $\sss$-space contains an uncountable $(1+\varepsilon)$-separated set for some $\varepsilon>0$.
\end{proposition}
\begin{proof}
Let $\A,\B$ be $T$-families inducing the  $\sw$-space $\X_\A$.
Consider a partial order $\PP$ consisting of finite sets $p\subseteq \omega_1$
satisfying $[p]^{<\omega}\cap \A\subseteq[p]^1\cup\{\emptyset\}$
with the order given by the inverse inclusion. 

Note that our hypothesis implies that $\PP$  does not satisfy c.c.c.
Indeed, if $\PP$ satisfies c.c.c., then (since $|\PP| = \omega_1$)  by Proposition \ref{ma-sigma} 
 $\PP = \bigcup_{n<\omega} \PP_n$ with $\PP_n$ centered. 
 Since $\{ \alpha \} \in \PP$ for every $\alpha<\omega_1$, 
 there would be $n\in \N$ and an uncountable
  $I\subseteq \omega_1$ such that $\{\alpha \} \in \PP_n$ for every $\alpha\in I$. 
  Then $I$ would be an uncountable set without any pair in $\A$ -- a contradiction with
  Definition \ref{def-T-families}.

Let $\PP'$ be an uncountable antichain in $\PP$, i.e., an uncountable family of finite subsets of $\omega_1$
 with no pair in $\A$ such that for every distinct $p, p'\in \PP'$ the union $p\cup p'$ has a pair in $\A$. 
 By applying the $\Delta$-system Lemma and removing the root we may assume that
they are pairwise disjoint and of the same finite cardinality. We will show that $\{\Car_p: p\in \PP'\}$
  forms an uncountable $\sqrt{2}$-equilateral
subset of the unit sphere of $\X_\A$. 
As the elements of $\PP'$  do not contain
 nonempty elements of $\A$ other than singletons,
it is clear that $\{\Car_p: p\in \PP'\}$ is included 
in the unit sphere of $\X_\A$. If $p, p'$ are distinct element of $\PP'$ and $A\in \A$, 
then $|A\cap (p\cup p')|\leq 2$. It follows that
$\|\Car_p-\Car_{p'}\|_{\A}\leq \sqrt{2}$. But as $p$ and $p'$ are incompatible, 
there must be $\xi\in p$ and $\eta\in p'$
such that $\{\xi, \eta\}\in\A$ and so 
$$\|\Car_p-\Car_{p'}\|_{\A}\geq\sqrt{|(\Car_p-\Car_{p'})(\xi)|^2+|(\Car_p-\Car_{p'})(\eta)|^2}=\sqrt{2},$$
which concludes the proof for $\sw$-spaces.

Now let $\D$ be a nice family with respect to $\A$ and $\B$ inducing the $\sss$-space $\X_{\A,\D}$.  
Since each $p\in \PP'$ is finite,  the value of  $\nu_\D(\Car_{p})$ is finite and so $\Car_{p}\in \X_{\A, \D}$. 
By passing to an uncountable set we may assume 
that there is $r\in \mathbb{R}$ such that $\nu(\Car_p) = r$ for every $p\in \PP'$
as the cardinality of $p$ limits the number of possible values of  $\nu_\D(\Car_p)$ to a finite set. 
Let us enumerate $\PP' = \{ p_\alpha:\ \alpha<\omega_1\}$.

For $\alpha<\omega_1$ let $D_\alpha\in \D$ be such
that 
$$\nu_\D(\Car_{p_\alpha})= \sqrt{\sum_{\{\xi, \eta\}\in D_\alpha}
|\Car_{p_\alpha}(\xi)-\Car_{p_\alpha}(\eta)|^2}.$$
Since $p_\alpha$s are pairwise disjoint, by passing to an uncountable subset we may assume that
for every $\alpha<\beta<\omega_1$ we have 
$$(\bigcup D_\alpha ) \cap p_\beta=\emptyset.$$
This implies that $\nu_\D(\Car_{p_\alpha})\leq\nu_\D(\Car_{p_\alpha}-\Car_{p_\beta})$.
It follows that
$$\| \Car_{p_\alpha} - \Car_{p_\beta} \|_{\A, \D} = 
\| \Car_{p_\alpha} - \Car_{p_\beta}\|_\A +\nu_\D(\Car_{p_\alpha} - \Car_{p_\beta}) \geq\sqrt 2+r.$$

So $\lbrace \frac{1}{1+r} \Car_{p_{\alpha}}:\ \alpha<\omega_1 \rbrace$ is 
an uncountable $\frac{\sqrt 2+r}{1+r}$-separated subset of the unit sphere of the $\sss$-space.
\end{proof}

\begin{lemma}\label{sqrt5} Suppose that $c=(c_0, c_1)\colon [\omega_1]^2\rightarrow 3\times[\omega_1]^{<\omega}$
is a $T$-coloring and    $\alpha_0<\alpha_1<\alpha_2<\alpha_3<\beta_0<\beta_1<\beta_2<\beta_3$ are  such that
\begin{enumerate}
\item $\{\alpha_1,\alpha_2\}, \{\beta_1,\beta_2\}\in \B_c$,
\item $\{\alpha_2,\beta_1\}\in \A_c$,
\item  $c(\{\alpha_1, \beta_2\})=(2,\emptyset)$,
\item $c(\{\alpha_1, \alpha_3\})=c(\{\alpha_0, \alpha_2\})=(2,\emptyset)$,
\item $c(\{\beta_1, \beta_3\})=c(\{\beta_0, \beta_2\})=(2,\emptyset)$,
\item $\alpha_1\not\in c_1(\{\alpha_2, \beta_2\})$.
\end{enumerate}
Then $\|\Car_{\{\alpha_1,\alpha_2\}}\|_{\A_c}=\|\Car_{\{\beta_1,\beta_2\}}\|_{\A_c}=1$,
$\nu_{\D_c}(\Car_{\{\alpha_1,\alpha_2\}})=\nu_{\D_c}(\Car_{\{\beta_1,\beta_2\}})=\sqrt2$,
$\|\Car_{\{\alpha_1,\alpha_2\}}-
\Car_{\{\beta_1,\beta_2\}}\|_{\A_c}=\sqrt2$ and $\nu_{\D_c}(\Car_{\{\alpha_1,\alpha_2\}}-\Car_{\{\beta_1\beta_2\}})=\sqrt5$.
\end{lemma}
\begin{proof}
$\|\Car_{\{\alpha_1,\alpha_2\}}\|_\A=\|\Car_{\{\beta_1,\beta_2\}}\|_\A=1$ is clear by (1).
Note that $\{\{\alpha_0, \alpha_1\}, \{\alpha_2, \alpha_3\}\}\in\D$
and $\{\{\beta_0, \beta_1\}, \{\beta_2, \beta_3\}\}\in\D$ by (4) and (5) and consequently 
$\nu_\D(\Car_{\{\alpha_1,\alpha_2\}})=\nu_\D(\Car_{\{\beta_1,\beta_2\}})=\sqrt2$.

If $A\in \A_c$ then $A\cap \{\alpha_1,\alpha_2\}, A\cap  \{\beta_1, \beta_2\}$  are at most singletons by (1).
So $|A\cap\{\alpha_1,\alpha_2, \beta_1,\beta_2\}|\leq 2$, and hence 
$\sqrt{\Sigma_{\alpha\in A}(\Car_{\{\alpha_1,\alpha_2\}}-
\Car_{\{\beta_1,\beta_2\}})(\alpha)^2)}$
is $0$ or $\sqrt{1^2}$ or $\sqrt{(-1)^2}$ or $\sqrt{(1)^2+(-1)^2}$. The third possibility takes place for 
$\{\alpha_2,\beta_1\}\in \A_c$ by (2). So $\|\Car_{\{\alpha_1,\alpha_2\}}-
\Car_{\{\beta_1, \beta_2\}}\|_{\A_c}=\sqrt2$.

Note that $\{\{\alpha_1, \beta_1\}, \{\beta_2, \beta_3\}\}\in \D_c$ since
$c(\{\alpha_1, \beta_2\})=(2,\emptyset)=c(\{\beta_1, \beta_3\})$ by (3) and (5).
Then $$((\Car_{\{\alpha_1,\alpha_2\}}-\Car_{\{\beta_1,\beta_2\}})(\alpha_1)-
(\Car_{\{\alpha_1,\alpha_2\}}-\Car_{\{\beta_1,\beta_2\}})(\beta_1))^2+$$
$$((\Car_{\{\alpha_1,\alpha_2\}}-\Car_{\{\beta_1,\beta_2\}})(\beta_2)-
(\Car_{\{\alpha_1,\alpha_2\}}-\Car_{\{\beta_1,\beta_2\}})(\beta_3))^2=$$
$$={(1-(-1))^2+(-1-0)^2}=5.$$
So $\nu_{\D_c}(\Car_{\{\alpha_1,\alpha_2\}}- \Car_{\{\beta_1,\beta_2\}})\geq\sqrt5$.
Consider $D\in\D$.
If a pair $\{\xi,\eta\}$ does not meet both $\{\alpha_1,\alpha_2\}$ 
and $\{\beta_1,\beta_2\}$, then 
$|(\Car_{\{\alpha_1,\alpha_2\}}- \Car_{\{\beta_1,\beta_2\}})(\xi)-(\Car_{\{\alpha_1,\alpha_2\}}- \Car_{\{\beta_1,\beta_2\}})(\eta)|$
is $1$ or $0$. So if  no pair of $D$ meets both $\{ \alpha_1, \alpha_2\}$ and $\{ \beta_1,\beta_2\}$, then
$$\sqrt{\Sigma_{\{\xi, \eta\}\in D}|(\Car_{\{\alpha_1,\alpha_2\}}- \Car_{\{\beta_1,\beta_2\}})(\xi)
-(\Car_{\{\alpha_1,\alpha_2\}}- \Car_{\{\beta_1,\beta_2\}})(\eta)|^2}$$ is not greater than
$\sqrt4=2$.
Since elements of $D\in \D$ are consecutive pairs, at most one $d\in D$
may meet both $\{\alpha_1,\alpha_2\}$ and $\{\beta_1,\beta_2\}$.  If only one other pair of $D$ meets 
$\{\alpha_1,\alpha_2, \beta_1,\beta_2\}$ then
$\sqrt{\Sigma_{\{\xi, \eta\}\in D}|(\Car_{\{\alpha_1,\alpha_2\}}- \Car_{\{\beta_1,\beta_2\}})(\xi)
-(\Car_{\{\alpha_1,\alpha_2\}}- \Car_{\{\beta_1,\beta_2\}})(\eta)|^2}$ is not bigger than
$\sqrt5$.  So we are left with proving that there cannot be $D\in \D$
with $d, d', d''\in D$ such that  $d<d'<d''$ and all $d, d', d''$ meeting $\{\alpha_1,\alpha_2, \beta_1,\beta_2\}$
and one of $d, d', d''$ meeting both $\{\alpha_1,\alpha_2\}$ and $\{\beta_1,\beta_2\}$.
This is only possible if $d'=\{\alpha_2, \beta_1\}$, $\alpha_1\in d$ and $\beta_2\in d''$. But then by the definition of $\D_c$
the pair $d$ must be included in the value of $c_1$ on the first elements of $d'$ and $d''$
and in the value of $c_1$ on the second elements of $d'$ and $d''$. So
$\alpha_1\in c_1(\alpha_2, \beta_2)$ or $\alpha_1\in c_1(\beta_1, \beta_2)$ and $c_0(\beta_1, \beta_2)=2$.
These possibilities are excluded by (1) and (6).

\end{proof}

\begin{proposition}\label{zfc-equi}
There is a $T$-coloring $c=(c_0, c_1)\colon [\omega_1]^2\rightarrow 3\times[\omega_1]^{<\omega}$
such that  $X_{\A_c}$ admits an uncountable $\sqrt2$-equilateral set
 and $\X_{\A_c, \D_c}$ admits an uncountable ${{\sqrt2+\sqrt5}\over{\sqrt2+1}}$-equilateral set.
\end{proposition}
\begin{proof} Let $e=(e_0, e_1)\colon [\omega_1]^2\rightarrow 3\times[\omega_1]^{<\omega}$ be any $T$-coloring.
Let $A_\xi$ be pairwise disjoint quadruples $\{\alpha_\xi^0, \alpha_\xi^1, \alpha_\xi^2, \alpha_\xi^3\}$ for $\xi<\omega_1$
satisfying $\alpha_\xi^0<\alpha_\xi^1<\alpha_\xi^2<\alpha_\xi^3<\alpha_\eta^0$ for any $\xi<\eta<\omega_1$.
For every $\xi<\eta<\omega_1$ we put 
\begin{enumerate}
\item $c(\alpha_\xi^1, \alpha_\xi^2)=(1, \emptyset)$,
 \item $c(\alpha_\xi^1, \alpha_\xi^3)=(2, \emptyset)$,
  \item $c(\alpha_\xi^0, \alpha_\xi^2)=(2, \emptyset)$,
  \item $c(\alpha_\xi^2, \alpha_\eta^2)=(1, \emptyset)$ if $\alpha_\xi^1\in e_1(\alpha_\xi^2, \alpha_\eta^2)$.
 \item $c(\alpha_\xi^1, \alpha_\eta^2)=(2, \emptyset)$,  
 \item $c(\alpha_\xi^2, \alpha_\eta^1)=(0, \emptyset)$,
\end{enumerate}
For all remaining pairs $\{\alpha, \beta\}\in [\omega_1]^2$ we put $c(\{\alpha, \beta\})=e(\{\alpha, \beta\})$.
 This way for every $\xi<\eta$ the sets $\{\alpha_\xi^0, \alpha_\xi^1, \alpha_\xi^2, \alpha_\xi^3\}$ satisfy the hypothesis
 of Lemma \ref{sqrt5} and so 
 $\|\Car_{\{\alpha^\xi_1, \alpha_\xi^2\}}\|_{\A_c, \D_c}=1+\sqrt2$
 $\|\Car_{\{\alpha^\xi_1, \alpha_\xi^2\}}-
 \Car_{\{\alpha^\eta_1, \alpha_\eta^2\}}\|_{\A_c, \D_c}=\sqrt2+\sqrt5$,  
 $\|\Car_{\{\alpha^\xi_1, \alpha_\xi^2\}}\|_{\A_c}=1$
 $\|\Car_{\{\alpha^\xi_1, \alpha_\xi^2\}}-
 \Car_{\{\alpha^\eta_1, \alpha_\eta^2\}}\|_{\A_c}=\sqrt2$ for every $\xi<\eta$ and so
 $\{\Car_{\{\alpha^\xi_1, \alpha_\xi^2\}}: \xi<\omega_1\}$ 
 is $\sqrt2$-equilateral  subset of the unit sphere in $\X_{\A_c}$ and 
 $\{{1\over{1+\sqrt2}}\Car_{\{\alpha^\xi_1, \alpha_\xi^2\}}: \xi<\omega_1\}$ is ${{\sqrt2+\sqrt5}\over{\sqrt2+1}}$-equilateral
 in $\X_{\A_c, \D_c}$.
 
 It remains to verify that $c$ is a $T$-coloring.  So suppose that
 $\{\{a_\gamma(0), a_\gamma(1)\}: \gamma<\omega_1\}$ 
 is pairwise disjoint and $(n, F), (m, G)\in 3\times[\omega_1]^{<\omega}$.
 By passing to an uncountable subfamily of $\{\{a_\gamma(0), a_\gamma(1)\}: \gamma<\omega_1\}$ 
 we may assume that 
 \begin{enumerate}
 \item[(a)] each set 
 $\{\alpha_\xi^0, \alpha_\xi^1, \alpha_\xi^2, \alpha_\xi^3\}$ intersects at most one
 pair $\{a_\gamma(0), a_\gamma(1)\}$
 \item[(b)]  $F\cup G<\{a_\gamma(0), a_\gamma(1)\}$ for every $\gamma<\omega_1$,
\item[(c)] if there are $\gamma, \xi<\omega_1$ such that
$a_\gamma(j)=\alpha_\xi^i$, for some $j\in \{0,1\}, i\in \{0,1,2,3\}$ 
then for every $\gamma<\omega_1$ there is $\xi<\omega_1$ such that
$a_\gamma(j)=\alpha_\xi^i$.
\end{enumerate}

Now use the hypothesis that $e$ is a $T$-coloring and 
find $\gamma<\gamma'<\omega_1$ such that  $e(\{a_\gamma(0), a_{\gamma'}(0)\})=(n, F)$
and  $e(\{a_\gamma(1), a_{\gamma'}(1)\})=(m, G)$.   We claim that  the coloring $e$ was not modified
on these pairs.  Indeed, because of (a) we did not modify by (1) or (2) or (3), because of (b) we did not modify
it by (4) and because of (c) we did not modify it by (5) or (6).
 So $c(\{a_\gamma(0), a_{\gamma'}(0)\})=(i, F)$
and  $c(\{a_\gamma(1), a_{\gamma'}(1)\})=(j, G)$, which proves that $c$ is a $T$-coloring.
 \end{proof}

\section{An $\sw$-space and a renorming of $\ell_2(\omega_1)$ from a strong $T$-coloring}

In this section we work with $\sw$-spaces $\X_\A$, where $\A$ is a strong $T_0$-family;
that is, $\A=\A_c$ for a strong $T$-coloring $c$ (see Definitions  \ref{T-coloring} and \ref{sT-coloring}).

\begin{lemma}\label{finite-support}
 Suppose that $\A$ is a strong $T_0$-family and $r\in \R$.
 Let
$\{x_\xi: \xi<\omega_1\}\subseteq c_{00}(\omega_1)$ consists of vectors with pairwise disjoint
supports and such that $\|x_\xi\|_\A=r$ for all $\xi<\omega_1$.  Then
\begin{itemize}
\item there are $\xi<\eta<\omega_1$ such that $\|x_\xi-x_\eta\|_\A=\sqrt{2}r$,
\item there are $\xi<\eta<\omega_1$ such that  $\|x_\xi-x_\eta\|_\A=r$.
\end{itemize}
\end{lemma}
\begin{proof} Let $F_\xi$ be the support of $x_\xi$ for
every $\xi<\omega_1$. For each $\xi<\omega_1$ there is
$A_\xi\in\A,\ A_\xi \subseteq F_\xi$ such that 
$\|x_\xi\|_\A=\sqrt{\Sigma_{\alpha\in A_\xi}|x_\xi(\alpha)|^2}$.

Since $\A$ is a strong $T_0$-family,  by Lemma \ref{sT-family} there are
$\xi<\eta<\omega_1$ such that $(F_\xi \otimes F_\eta) \cap \A = \emptyset$.
Note that in this case if $A\in \A$ then either $A\cap F_\xi=\emptyset$
or $A\cap F_\eta=\emptyset$, so
$\|x_\xi-x_\eta\|_\A=\|x_\xi\|_\A=\|x_\eta\|_\A=r$.

Since $\A$ is a strong $T_0$-family,  by Lemma \ref{sT-family} there are also
$\xi<\eta<\omega_1$ such that $A_\xi \otimes A_\eta \subseteq \A$.
As $A_\xi, A_\eta\in \A$, it follows  that $A_\xi\cup A_\eta\in \A$
and so $\|x_\xi-x_\eta\|_\A=\sqrt2 r$.
 
\end{proof}

\begin{proposition}
Let $\A$ be a strong $T_0$-family. For every $\delta>0$ there is $\varepsilon>0$ such that for every uncountable $(1-
\varepsilon)$-separated $\{x_\alpha:\ \alpha<\omega_1\} \subseteq S_{\X_\A}$ there are $
\alpha<\beta <\omega_1$ such that $\|x_\alpha - x_\beta\|_\A > \sqrt{2} - \delta$ and 
there are $\xi<\eta<\omega_1$ such that $\|x_\xi - x_\eta\|_\A < 1+\delta$.
\end{proposition}
\begin{proof}
Fix $\delta>0$ and consider $\varepsilon< \frac{\delta}{2\sqrt{2}+1}$ and an uncountable $(1-
\varepsilon)$-separated set $\{x_\alpha:\ \alpha<\omega_1\} \subseteq S_{\X_\A}$. 

For every $\alpha<\omega_1$ let us pick $y_\alpha \in B_{\X_\A}$, $y_\alpha\in c_{00}
(\omega_1)$ with 
rational values such that $\|x_\alpha - y_\alpha\|_\A \leq \frac{\varepsilon}{2}$ . By 
passing to an uncountable set we may assume that $\{ \supp(y_\alpha):\ \alpha<\omega_1\}$ is a $\Delta$-system 
with its root $\Delta$. By passing to a further uncountable set we may assume that $y_\alpha(\delta) = y_\beta(\delta)$
 for every $\delta\in \Delta$ and $\alpha<\beta<\omega_1$. In particular, if $z_\alpha = 
P_{\omega_1\setminus \Delta} (y_\alpha)$, then $\{ \supp(z_\alpha):\ \alpha<\omega_1\}$ are 
pairwise disjoint and $\|z_\alpha-z_\beta\|_\A = \|y_\alpha - y_\beta\|_\A \geq 1-2\varepsilon$. 
As the values of all $z_\alpha$s are
rational we may assume that $\|z_\alpha\|_\A =r$ for some $r\in\R$ and
 for every $\alpha<\omega_1$. 

By Lemma \ref{finite-support} there are $\alpha<\beta <\omega_1$ such that 
$\| y_\alpha - y_\beta\|_\A = \|z_\alpha - z_\beta\|_\A = \sqrt{2}r$
and there are $\xi<\eta<\omega_1$ such that 
$\| y_\xi - y_\eta \|_\A = \|z_\xi - z_\eta\|_\A = r.$
In particular, $r\geq 1-2\varepsilon$ since $\{y_\alpha:\ \alpha<\omega_1\}$ is 
$(1-2\varepsilon)$-separated.
It follows that 
$$\| x_\alpha - x_\beta \|_\A \geq \sqrt{2}r - \varepsilon 
\geq \sqrt{2}(1-2\varepsilon) - \varepsilon = \sqrt{2} - \varepsilon(2\sqrt{2} +1)> \sqrt{2} - \delta.$$
By Lemma \ref{projection-norm} we have $r=\|z_\xi\|_\A\leq\|y_\xi\|_\A \leq 1$ and so
$$\| x_\xi - x_\eta\|_\A \leq r + \varepsilon \leq 1 + \varepsilon < 1 +\delta,$$
which concludes the proof.
\end{proof}

Since Banach spaces admitting uncountable equilateral sets admit
$1$-equilateral sets in their unit sphere, the previous proposition immediately yields the following.

\begin{proposition}\label{equi-sep}
Suppose that $\A$ is strong $T_0$-family and $\varepsilon>0$.
Then $\X_\A$ does not admit any uncountable equilateral set and
 $S_{\X_\A}$ doesn't admit any uncountable $(1+\varepsilon)$-separated set.
\end{proposition}

\begin{proposition}\label{somewhat} If $\A$ is a strong $T_0$-family, then every nonseparable subspace of $\X_\A$
contains a subspace isomorphic to $\ell_2$.
\end{proposition}
\begin{proof} Let $\Y$ be a nonseparable subspace of $\X_\A$.  By Proposition \ref{hilbert}
the space $\X_\A$ is a Hilbert generated Banach space and hence a WLD Banach space and 
so is $\Y$ (Corollary 9 of \cite{wld}). It follows that there is $\Y'\subseteq \Y$
which is linearly dense  in $\Y$ such that $\{y\in \Y': y^*(y)\not=0\}$ is at most countable for each $y^*\in \Y^*$ (Theorem
7 (iv) of \cite{wld}).
It follows that for every $\alpha<\omega_1$ there is $y\in \Y'\setminus\{0\}$ such that
$y(\beta)=0$ for every $\beta<\alpha$.
As the support of any vector in $\X_\A$ is countable,  this allows us to inductively construct 
nonzero vectors $y_\xi\in \Y'\subseteq\Y$ for $\xi<\omega_1$ which have pairwise disjoint supports.

Let $A_\xi$ be the support of $y_\xi$.  By normalizing $y_\xi$ we may assume that
$y_\xi$s have norm one. Find finite $B_\xi\subseteq A_\xi$ such that  $B_\xi\in \A$ and 
 $$\sqrt{\sum_{\alpha\in B_{\xi}}(y_{\xi}(\alpha))^2}\geq 1/2.$$
Consider a coloring $c\colon[\omega_1]^2\rightarrow \{0,1\}$ defined by $c(\{\xi, \eta\})=0$ if and only if
$B_\xi\cup B_\eta\in \A$.  By the Erd\H{o}s-Dushnik-Miller
theorem and the properties of a strong $T_0$-family (see Lemma \ref{sT-family})
there is an infinite $X\subseteq\omega_1$ which is $0$-monochromatic for $c$.
This implies that all finite subsets of $\bigcup\{B_\xi: \xi\in X\}$ are in $\A$.

We claim that there is an isomorphism $T\colon\ell_2\rightarrow \Y$ such that $T(\Car_{\{n\}})=y_{\xi_n}$,
where $\xi_n$ is the $n$-th element of $X$.   For this it is enough to prove
that 
$${1\over 2}\sqrt{\sum_{n<m}c_n^2}\leq \|\sum_{n<m}c_n y_{\xi_n}\|_\A\leq \sqrt{\sum_{n<m}c_n^2}$$
for any choice of $c_0, \dots c_{m-1} \in \R$. 
For the first inequality we have
$$\|\sum_{n<m}c_n y_{\xi_n}\|_\A\geq \sqrt{\sum_{n<m}\sum_{\alpha\in B_{\xi_n}}(c_n y_{\xi_n}(\alpha))^2}\geq
\sqrt{\sum_{n<m}{1\over 4}c_n^2}={1\over 2}\sqrt{\sum_{n<m}c_n^2}$$
as required. For
the  second inequality first note that for every $\varepsilon>0$ there is $A\in \A$
such that  $\|\sum_{n<m}c_ny_{\xi_n}\|^2_\A\leq \sum_{\alpha\in A}\big(\sum_{n<m}c_ny_{\xi_n}(\alpha)\big)^2+\varepsilon$.
Since $y_{\xi_n}$s have disjoint supports for a given $\alpha\in A$ the
sum $\sum_{n<m}c_ny_{\xi_n}(\alpha)$ have at most one nonzero summand, and hence
for $\supp(y_{\xi_n})\cap A=C_n\in\A$  we  have
 $$\sum_{\alpha\in A}\big(\sum_{n<m}c_ny_{\xi_n}(\alpha)\big)^2=
 \sum_{n<m}\big(\sum_{\alpha\in C_n}(c_ny_{\xi_n}(\alpha))^2\big)\leq
  \sum_{n<m}\|c_ny_{\xi_n}\|^2_\A.$$
As $\varepsilon$ was arbitrary, we have $\|\sum_{n<m}c_ny_{\xi_n}\|^2_\A\leq\sum_{n<m}\|c_ny_{\xi_n}\|^2_\A$
which gives the second inequality.
\end{proof}

\begin{proposition}\label{hilbert-equi} Suppose that $\A$ is a strong $T_0$-family
and $\|\ \|_\n$ is as in Proposition \ref{hilbert}. Then 
the space $(\ell_2(\omega_1), \|\ \|_\n)$ doesn't contain uncountable equilateral sets.
\end{proposition}
\begin{proof}
Note that  it follows from Lemma \ref{projection-norm} that $P_A$ for $A\subseteq\omega_1$ is a projection 
of norm one on the space $(\ell_2(\omega_1), \|\ \|_\n)$.
It is enough to show that the unit sphere $S$ of this 
space doesn't contain such a set. We will show that for every 
non-separable $\Y\subseteq S$ 
there are $x_1, x_2, x_3, x_4\in \Y$ such that $\|x_1 - x_2\|_{2,\A} < \|x_3 - x_4\|_{2,\A}$.

Using the nonseparability of $\Y$ construct a strictly increasing $(\gamma_\alpha:\alpha<\omega_1)$
in $\omega_1$ and $\{x_\alpha: \alpha<\omega_1\}\subseteq \Y$ such that $x_\alpha(\gamma_\alpha)\not=0$
for every $\alpha<\omega_1$. By passing to an uncountable set we may assume that
$|x_\alpha(\gamma_\alpha)|>\varepsilon$ for every $\alpha<\omega_1$.
 
For every $\alpha<\omega_1$ let $y_\alpha \in \ell_2(\omega_1)$ 
be a vector with a finite support and rational coefficients such that 
$\|x_\alpha - y_\alpha\|_{2,\A} < \frac{\varepsilon}{20}$ and $|y_\alpha(\gamma_\alpha)|> \varepsilon$.

By the $\Delta$-system Lemma we  may assume that 
$\{\supp(y_\alpha):\ \alpha<\omega_1\}$ forms a $\Delta$-system with a root $\Delta$
and $\gamma_\alpha > \max \Delta$ for every $\alpha<\omega_1$ as $\gamma_\alpha$s are distinct. 
Let $z_\alpha = P_{\omega_1 \setminus \Delta}(y_\alpha)$. 
Then $z_\alpha$s have rational coefficients and 
 pairwise disjoint finite supports containing $\gamma_\alpha$ and so $\| z_\alpha \|_\A>\varepsilon$.
Once again by passing to an uncountable set we may assume that
 $\|z_\alpha\|_2=R$ and $\|z_\alpha\|_\A =r$ for every $\alpha<\omega_1$ 
 and some $r, R\in \R$, necessarily $R\geq r > \varepsilon$.
 Let us note that for every $\alpha<\beta <\omega_1$ we have 
$$\| y_\alpha - y_\beta\|_\n = \|z_\alpha - z_\beta\|_\n = \| z_\alpha - z_\beta \|_2 
+ \|z_\alpha - z_\beta\|_\A = \sqrt{2} R + \| z_\alpha - z_\beta\|_\A,$$
since the supports of $z_\alpha$ and $z_\beta$ are disjoint and $\|z_\alpha\|_2 = \|z_\beta\|_2 = R$.
Lemma \ref{finite-support} yields that there are $\alpha<\beta < \omega_1$ and   $\xi<\eta<\omega_1$ such that 
$$\|y_\alpha - y_\beta\|_{2,\A} = \sqrt{2}R + \sqrt{2}r \ \ \hbox{and}\ \ 
\|y_\xi - y_\eta\|_{2,\A} = \sqrt{2}R + r.$$
It follows from  the choice of $y_\alpha$s  that
$$\|x_\alpha - x_\beta\|_{2,\A} \geq \sqrt{2}R + \sqrt{2}r - \frac{\varepsilon}{10} \ \ \hbox{and}\ \ 
\|x_\xi - x_\eta\|_{2,\A} \leq \sqrt{2}R + r + \frac{\varepsilon}{10}.$$
However, 
$$\sqrt{2}r - \frac{\varepsilon}{10} > r + \frac{4}{10} r - \frac{\varepsilon}{10} > r
 + \frac{3\varepsilon}{10} > r + \frac{\varepsilon}{10},$$
from which it follows that
$\|x_\xi - x_\eta\|_{2,\A} < \|x_\alpha - x_\beta\|_{2,\A},$
which concludes the proof.
\end{proof}

\section{Questions}

\begin{question}$ $
\begin{enumerate}
\item Is there an infinite dimensional reflexive Banach space without an infinite equilateral set?
\item Is there a nonseparable reflexive (Hilbert generated, WCG, WLD) Banach space without an uncountable equilateral set?
\item Is there an equivalent renorming of $\ell_2(\omega_1)$ without an uncountable equilateral set?
\item Is there a nonseparable Hilbert generated (WCG, WLD) Banach space whose unit sphere does not admit
an  uncountable $(1+)$-separaed set?
\item Is there an equivalent renorming of $c_0(\omega_1)$ without an uncountable equilateral set?
\end{enumerate}
\end{question}
Of course items (1) - (3) ask for absolute examples or for proofs of the consistency of
the negative answers as our Theorem \ref{main} provides consistent positive answer to these questions.
It should be stressed that this does not exclude 
the possibility that a particular absolute example of an $\sss$-space with some extra properties
yields the positive answer to items (1) - (3). An $\sw$-space cannot serve here by Theorem \ref{main-ssw} (3).
More delicate constructions of norms of nonseparable spaces with few operators
than in \cite{ss}  like, for example, in \cite{stevo-jordi} could be relevant here.

However, already a consistent positive answer to items (4) or (5) would be interesting and new.  
The only examples of nonseparable Banach spaces whose unit spheres do not admit
 uncountable $(1+)$-separaed set are presented in \cite{pk-kottman, pk-ad} and are far from being WLD as
 their duals are weakly$^*$ separable. In fact their spheres do not admit even any uncountable
 $1$-separated set and as a consequence such spaces do not contain uncountable Auerbach systems
 nor uncountable equilateral sets. 
Assuming the continuum hypothesis an interesting
consistent equivalent renorming  of $c_0(\omega_1)$  which does not
admit any uncountable Auerbach system was constructed in \cite{hkr}. This should be compared with the fact 
that every equivalent renorming of $c_0$ admits an infinite equilateral set (\cite{mer-pams}).

\bibliographystyle{amsplain}

\end{document}